\documentclass[12pt,leqno]{article}
\usepackage[utf8]{inputenc}
\usepackage{mathrsfs}
\usepackage{indentfirst,amsmath,amsfonts,amsthm}

\def\ll{L^2_xL^2_M}

\def\lbarl{\overline{L}^2_xL^2_M}
\newcommand\hl[1]{H^{#1}_xL^2_M}

\newcommand\hbarl[1]{\overline{H}^{#1}_xL^2_M}
\newcommand\hh[1]{H^{#1}_xH^1_M}
\newcommand\hdoth[1]{H^{#1}_x\dot{H}^1_M}

\newcommand\hbardoth[1]{\overline{H}^{#1}_x\dot{H}^1_M}
\def\ldoth{L^2_x\dot{H}^1_M}

\newcommand\normll[1]{\|#1\|_{\ll}}

\newcommand\normlbarl[1]{\|#1\|_{\lbarl}}
\newcommand\normldoth[1]{\|#1\|_{\ldoth}}

\newcommand\normhl[2]{\|#1\|_{\hl{#2}}}
\newcommand\normhbarl[2]{\|#1\|_{\hbarl{#2}}}

\newcommand\normhdoth[2]{\|#1\|_{\hdoth{#2}}}

\newcommand\normhbardoth[2]{\|#1\|_{\hbardoth{#2}}}

\def\f{f}
\def\ft{\widetilde{f}}
\def\ep{\varepsilon}

\def\al{\alpha}
\DeclareMathOperator\dive{div}
\numberwithin{equation}{section}
\newcommand{\scal}[2]{\langle {#1},{#2}\rangle}
\def\bq{\overline{q}}
\def\bM{Z}
\def\zbar{\overline{Z}}
\newcommand{\ud}{\mathrm{d}}
\newcommand\ddt{\frac{\ud}{\ud t}}

\def\R{\mathbb{R}}
\def\C{\mathbb{C}}
\def\S{\mathbf{S}}

\def\B{\mathcal{B}}
\def\A{\mathbf{A}}
\def\str{\mathbf{J}}
\def\cgrad{\cdot\nabla }
\newcommand{\PP}{\mathbb{P}}
\def\dq{\,\ud q}
\def\dx{\,\ud x}

\def\ds{\,\ud s}

\def\psib{\overline{\psi}}

\def\ub{\overline{u}}
\newtheorem{theorem}{Theorem}
\newtheorem{proposition}[theorem]{Proposition}
\newtheorem{lemma}[theorem]{Lemma}

\newtheorem{remark}[theorem]{Remark}

\date{}
\title{The FENE dumbbell polymer model: existence and uniqueness of solutions for the momentum balance equation.}
\author{A.V. Busuioc, I.S. Ciuperca, D. Iftimie and L.I. Palade}
\def\adrese{
\begin{description}
\item[Adriana Valentina Busuioc:] Université Jean Monnet -- Faculté des Sciences, LaMuse, 23 Rue du Docteur Paul Michelon, 42023 Saint-Etienne, France.\\
Email: \texttt{valentina.busuioc@univ-st-etienne.fr}
\item[Ionel Sorin Ciuperca:] Université de Lyon, CNRS, Université Lyon 1, Institut Camille Jordan, 43 bd. du 11 novembre, Villeurbanne Cedex F-69622, France.\\
Email: \texttt{ciuperca@math.univ-lyon1.fr}
\item[Dragoş Iftimie:] Université de Lyon, CNRS, Université Lyon 1, Institut Camille Jordan, 43 bd. du 11 novembre, Villeurbanne Cedex F-69622, France.\\
Email: \texttt{iftimie@math.univ-lyon1.fr}\\
Web page: \texttt{http://math.univ-lyon1.fr/\~{}iftimie}
\item[Liviu Iulian Palade:] Université de Lyon, CNRS, INSA-Lyon,  Institut Camille Jordan \& Pôle de Mathématiques, Bât. Leonard de Vinci No. 401, 21 avenue Jean Capelle, F-69621, Villeurbanne, France.\\
Email: \texttt{Liviu-Iulian.Palade@insa-lyon.fr}
\end{description}
}

\begin{document}
\maketitle
\begin{abstract}
We consider the FENE dumbbell polymer model which is the coupling of the incompressible Navier-Stokes equations with the corresponding Fokker-Planck-Smoluchowski diffusion equation. We show global well-posedness in the case of a 2D bounded domain. We assume in the general case that the initial velocity is sufficiently small and the initial probability density is sufficiently close to the equilibrium solution; moreover an additional condition on the coefficients is imposed. In the corotational case, we only assume that  the initial probability density is sufficiently close to the equilibrium solution.
\end{abstract}

\section{Introduction}

The success of Kirkwood, and of Bird, Curtiss, Armstrong and
Hassager (and their collaborators) kinetic theory of macromolecular
dynamics triggered a still on-going flurry of activity aimed to
providing molecular explanations for non-Newtonian and viscoelastic
flow patterns. This can be reckoned from \cite{BAH87} and \cite{Ott06}, for example.  The cornerstone is the so called diffusion equation, a parabolic-type Fokker-Planck-Smoluchowski partial differential equation, the solution of which is the configurational probability distribution function; the later is the key ingredient for calculating the stress tensor.

The simplest polymer chain model of relevance to Bird {\it et al.} 
theory is that of a dumbbell, where the beads are interconnected either rigidly
or elastically. Although a crude representation of the complicated
 dynamics responsible for the flow viscoelasticity, the
now popular Bird and Warner's 
Finitely Extensible Nonlinear Elastic (FENE for the short; see \cite{War72}) chain model
is capable in capturing many salient experimentally observable flow
patterns of dilute polymer solutions. It was therefore quite natural
that many researchers took on exploring the fundamentals of this
relatively simple model (for more on this and related issues see for example \cite{BE94} and \cite{Sch06}).

The aim of this work is to take on studying the momentum-balance 
(or Navier-Stokes) equations together with the constitutive law for 
the FENE fluid. The latest is obtained by using the so-called
``diffusion equation'', 
practically a Fokker-Planck PDE, the solution of which is the 
configurational probability density. Put it differently, we focus 
on a system of equations that consists of a 
``macroscopical'' motion PDE and a ``microscopical'' Fokker-Plank-Smoluchowski (probability diffusion) PDE. More precisely, given a smooth bounded connected open set $\Omega\subset \R^d$ and some ball $D(0,R)$ we will study the initial boundary value problem  which consists in finding  $u=u(t,x):\R_+\times\Omega\rightarrow\mathbb{R}^d$,  $g=g(t,x,\bq):\R_+\times\Omega\times D(0,R)\rightarrow\mathbb{R}$  and $p=p(t,x):\R_+\times\Omega\rightarrow\mathbb{R}$ solutions of the two following coupled equations:
\begin{equation}\label{eq1}
 \partial_t u+u\cdot\nabla
u-\dfrac{\gamma} {\text{Re}}\triangle u+\nabla
p=\dfrac{\gamma(1-\gamma)}{\text{Re} \, \text{We}^2}\nabla_x \cdot \left(
  \int_{D(0,R)}\dfrac{\bq \otimes\bq }{1-\frac{|\bq |^2}{R^2}}g(t,x,\bq )
  \ud\bq  \right) \quad \text{on } \R_+\times\Omega
\end{equation}
and 
\begin{equation}\label{eq2}
 \partial_t g+
u\cdot\nabla_{x}g+\nabla_{\bq }\cdot \left(
  \sigma(u)\bq g \right)
=
\frac{1}{2 \text{We} \, N}\triangle_{\bq }g+\dfrac{1}{2 
\text{We}} \nabla_{\bq }\left( \dfrac{\bq }{1-\frac{|\bq |^2}{R^2}}g \right)
\quad \text{on } \R_+\times\Omega\times D(0,R). 
\end{equation}
Moreover, the vector field $u$ must be divergence free and $g$ must be a probability density in the $\bq$ variable:
\begin{equation}\label{eq3}
\dive_xu=0,\quad \int_{D(0,R)} g \,\ud\bq \equiv1, \quad g\geq 0. 
\end{equation}
The boundary conditions are
\begin{equation}\label{eq4}
u\bigl|_{\partial\Omega}=0
\end{equation}
plus some boundary conditions for $g$ on $\Omega\times\partial D(0,R)$ which will be embedded in the function spaces we will work with.

The constant $\gamma$ belongs to $(0,1)$, Re and We  are (respectively) the
Reynolds and
Weissenberg numbers and  $N$, $R$ are some polymer related physical
constants used to obtain dimensionless quantities.  We assume all these constants to be strictly positive and moreover that $NR^2>2$. The quantity $\sigma(u)$ is a
short-hand notation for either $\nabla u$ or $\nabla u-\left( \nabla
  u\right) ^t$. In fact, the physical significance is achieved when $\sigma(u)=\nabla u$; we will call this the general case. The choice $\sigma(u)=\nabla u-\left( \nabla u\right) ^t$ is very close to being physical significant while having better mathematical properties; we will call this the corotational case.
Let
\begin{equation*}
 \bM(\bq) = \left(1-\frac{|\bq |^2}{R^2} \right) ^{NR^2/2}\qquad\text{and}\qquad \zbar=\frac{\bM}{\int_{D(0,R)}\bM}\cdot 
\end{equation*}
It is not hard to observe that the couple $(0,\zbar)$ is a steady solution of \eqref{eq1}--\eqref{eq4}.  

The initial boundary value problem \eqref{eq1}--\eqref{eq4} was studied by several authors but mostly in the case where $\Omega=\R^2$ or $\R^3$. The results are different, depending on the model (general or corotational). We start by describing the results where  $\Omega=\R^2$ or $\R^3$. We restrict ourselves to the model described above, but we would like to mention that there are other results on closely related problems (for example a model when the variable $\bq$ lies in the full plane or full space, the Hookean model, etc.). We refer to \cite{Mas-arxiv-2010} for a discussion of all these models.

Global existence and uniqueness of strong solutions of problem \eqref{eq1}--\eqref{eq4} is known in the following situations:
\begin{itemize}
\item $\Omega=\R^2$ and corotational model if $u_0\in H^s(\R^2)$ and $g_0\in H^s(\R^2;H^1_0(D(0,R)))$, $s>2$ (see \cite{LZZ08}). The regularity of $g_0$ in the $\bq$ variable was improved in \cite{Mas08a} to some $L^p$ weighted space for large $p$.
\item  $\Omega=\R^2$ and general model or   $\Omega=\R^3$ and general or corotational model if $u_0$ is small in  $H^s(\R^2)$ and if $\bigl\|Z^{-\frac12}\|g_0-\zbar\|_{H^s(\R^2)}\bigr\|_{L^2(D(0,R))}$ is small, where $s>1+\frac d2$ where $d\in\{2,3\}$ is the space dimension (see \cite{LZ08,Mas08a}). 
\end{itemize}


Global existence (no uniqueness yet) of some weak solutions for rough and arbitrarily large initial data was proved in both dimension 2 and 3, first in the corotational case by \cite{LM07} and quite recently in the general case by \cite{Mas-arxiv-2010}, see also \cite{BS-arxiv-2010} for a slightly different version of the system of equations.

Long time asymptotics of the general model were studied in  \cite{JBLO06} where a priori estimates are obtained to prove  formally the stability of the equilibrium solution.  In \cite{BSS05} the authors studied a related model where a smoothing operator is acting on the velocity field and the corresponding stress tensor.

As far as strong solutions on domains with boundaries are concerned, we are aware of two works. One is \cite{ZZ06} where local existence and uniqueness is proved if $u_0\in H^4(\Omega)$ and if $g_0$ is  $H^4$ in $x$ and has some weighted $H^3$ regularity in the variable $\bq$. Another one is \cite{KP10} where local existence is shown if  $u_0\in W^{1,p}(\Omega)$ and $g_0$ is $W^{1,p}$ in $x$ and has some weighted $L^p$ regularity in the $\bq$ variable and $p>d$.

The goal of this paper is to address the issue of existence and uniqueness of strong solutions for the above mentioned initial boundary value problem on bounded domains $\Omega \subset \mathbb{R}^2$ with homogeneous Dirichlet boundary conditions. This is not a straightforward adaptation of the known results in the full plane. Indeed, the proof of global existence results of solutions proved by \cite{LZZ08} uses heavily the Littlewood-Paley decomposition and paradifferential calculus; this is of course not available on bounded domains. Even the global existence results for small data involve technical difficulties that make necessary to assume an additional condition of the material coefficients, more precisely we will need to assume \eqref{condcoefftheo}. We refer to Section \ref{finalremarks} for a detailed explanation why this is necessary.


In the general case, we show the following global existence and uniqueness result for initial data which is sufficiently close to the equilibrium solution $(0,\zbar)$.
\begin{theorem}[general case]\label{thgen}
Let $s\in (1,\frac32)$. Assume that $u_0$ is divergence free, vanishes on $\partial\Omega$ and belongs to $H^s(\Omega)$. Assume moreover that $\bM^{-\frac12}\|g_0\|_{H^s(\Omega)}\in L^2(D(0,R))$, $g_0\geq0$ and $\int_{D(0,R)}g_0\ud\bq\equiv1$. There exists two positive constants $K_1=K_1(\Omega,s)$ and $K_2=K_2(\Omega,s,\gamma,\text{\rm Re},\text{\rm We},N, R)$ such that if the fluid related coefficients verify the relation
\begin{equation}\label{condcoefftheo}
\frac{1-\gamma}{N\text{\rm We}}\leq K_1
\end{equation}
and if the initial data is sufficiently close to the equilibrium solution $(0,\zbar)$
\begin{equation*}
\|u_0\|_{H^s(\Omega)}\leq K_2\quad\text{and}\quad  \Bigl\|\frac{\|g_0-\zbar\|_{H^s(\Omega)}}{\sqrt{\bM}}\Bigr\|_{L^2(D(0,R)}\leq K_2
\end{equation*}
then there exists a unique solution to system \eqref{eq1}--\eqref{eq4} such that
\begin{gather*}
u\in L^\infty(\R_+;H^s(\Omega))\cap L^2(\R_+;  H^{s+1}(\Omega))\\
\intertext{and}
\bigl\|\bM^{-\frac12}\|g\|_{H^s(\Omega)}\bigr\|_{L^2(D(0,R)}\in L^\infty(\R_+),\quad
\bigl\|\bM^{\frac12}\|\nabla_{\bq}(g/\bM)\|_{H^s(\Omega)}\bigr\|_{L^2(D(0,R)}\in L^2(\R_+).
\end{gather*}
\end{theorem}

In the corotational case we improve the previous result in the following manner. Not only the restriction on the material coefficients \eqref{condcoefftheo} is no longer required, but the initial velocity $u_0$ is arbitrarily large as well. More precisely, we have the following theorem.
\begin{theorem}[corotational case]\label{thcor}
Let $s\in (1,\frac32)$. Assume that $u_0$ is divergence free, vanishes on $\partial\Omega$ and belongs to $H^s(\Omega)$. Assume moreover that $\|g_0\|_{H^s(\Omega)}/\sqrt{\bM}\in L^2(D(0,R)$, $g_0\geq0$ and $\int_{D(0,R)}g\ud\bq\equiv1$. There exists a positive constant $K_3=K_3(\Omega,s,\gamma,\text{\rm Re},\text{\rm We},N, R)$ such that if the following smallness assumption holds true
\begin{equation*}
\Bigl\|\frac{\|g_0-\zbar\|_{H^s(\Omega)}}{\sqrt{\bM}}\Bigr\|_{L^2(D(0,R)}\leq \exp\Bigl[-K_3(1+\|u_0\|_{H^s(\Omega)}) e^{K_3\|u_0 \|_{L^2(\Omega)}^{\frac4s}}\Bigr],
\end{equation*}
then there exists a unique solution to system \eqref{eq1}--\eqref{eq4} such that
\begin{gather*}
u\in L^\infty(\R_+;H^s(\Omega))\cap L^2(\R_+;  H^{s+1}(\Omega))\\
\intertext{and}
\bigl\|\bM^{-\frac12}\|g\|_{H^s(\Omega)}\bigr\|_{L^2(D(0,R)}\in L^\infty(\R_+),\quad
\bigl\|\bM^{\frac12}\|\nabla_{\bq}(g/\bM)\|_{H^s(\Omega)}\bigr\|_{L^2(D(0,R)}\in L^2(\R_+).
\end{gather*}
\end{theorem}

Compared to the result of \cite{LZZ08} valid in the case of the full plane, we have an additional condition on $g_0$: it needs to be close to $\zbar$. As explained above, this is due to the fact that we work with bounded domains and the methods of \cite{LZZ08} do not work here. Nevertheless, we have an improvement in the regularity assumptions. More precisely, we require a regularity in the $x$ variable which is $H^s$, $1<s<\frac32$, while in \cite{LZZ08,Mas08a} it is necessary to assume that $s>2$.  The regularity in the $\bq$ variable is also improved, roughly from $H^1$ to $L^2$.

The paper is organized as follows. In Section \ref{sbp}
we reformulate the problem \eqref{eq1}--\eqref{eq4} and introduce the notations. We construct next in Section \ref{as} a sequence of approximate solutions. The global existence of the approximate solutions is proved in Section \ref{approxex}. We show uniform estimates for the approximate solutions and complete the proofs of Theorems \ref{thgen} and \ref{thcor} in Section \ref{unif}. The last section contains two final remarks on the hypothesis we have to assume.

\section{Notations and functional framework}
\label{sbp}

We start by making a change of functions allowing to rewrite the equations in a better form. Notice first that
\begin{equation*}
 \nabla_{\bq } g+N
\dfrac{\bq }{1-\frac{|\bq |^2}{R^2}}g=\bM
\nabla_{\bq } \left(\frac g{\bM} \right).
\end{equation*}

If we set
\begin{gather*}
q=\frac{\bq}R, \quad  M(q)=(1-|q|^2)^\delta,\quad \f(t,x,q)=g(t,x,Rq)\\
\intertext{and}
\al_1= \dfrac \gamma {Re},\quad 
\al_2= \dfrac {\gamma(1-\gamma)} {\text{Re} \, \text{We}^2}
\left ( \dfrac {2 \delta} N \right )^{2},\quad
\al_3= \dfrac 1 {4\delta \, \text{We}}, \quad
\delta= \dfrac {NR^2} 2,
\end{gather*}
then the couple $(u,\f)$ must verify the following system of equations:
\begin{gather}
 \partial_t u+u\cdot\nabla u-\al_1\triangle u+
\nabla p
=\al_2 \nabla_x \cdot \int_D \dfrac{q\otimes
  q}{1-|q|^2}\f 
\dq\quad \text{on } \R_+\times\Omega\label{fequ}\\
 \partial_t \f+u\cdot\nabla_x \f-\al_3
\nabla_q\cdot \left[M\nabla_q\left(\dfrac{\f}{M} \right)  \right]+
\nabla_q\cdot \left(\sigma(u)q\f \right)
=0\label{feqpsi}\\
\hskip 13 cm  \text{on } \R_+\times\Omega\times D(0,1)\nonumber\\
\intertext{and}
\dive u=0,\qquad \int_{D(0,1)} \f \dq\equiv\frac1{R^2}. \nonumber
\end{gather}
The boundary conditions are homogeneous Dirichlet conditions for the velocity $u$ plus some boundary conditions for $f$ on $\Omega\times \partial D(0,1)$ which are implicit from the condition that $f$ has the below defined $H^1_M$ regularity in the $q$ variable (we refer to \cite{Mas08a} for a discussion on the boundary conditions verified by functions in $H^1_M$).
We also prescribe the initial data $u\bigl|_{t=0}=u_0$ and $\f\bigl|_{t=0}=\f_0$. 

\bigskip

We recall now the usual spaces for the solutions of
the Navier-Stokes equations:
$$
H = \{ v \in (L^2(\Omega))^2; \; \; \nabla \cdot v = 0, \; v \cdot \nu =
0 \; \; \text { on } \; \partial \Omega \}
  $$
where $\nu$ is the outward normal to $\partial \Omega$ 
and
$$
V = \{ v \in (H_0^1(\Omega))^2; \; \; \nabla \cdot v = 0 \}.
  $$
We will abbreviate in the following $D=D(0,1)$. We introduce next the following Banach spaces:
\begin{gather*}
L^2_M=L^2(D,\frac1M\dq)=\bigl\{\varphi;\ \int_D\frac{\varphi^2}M\dq<\infty\bigr\},\\
 \ll=L^2(\Omega\times D,\frac1M\dx\dq)=\bigl\{\varphi;\ \iint_{\Omega\times D}\frac{\varphi^2}M\dx\dq<\infty\bigr\},\\ 
\lbarl=L^2(\R^2\times D,\frac1M\dx\dq)=\bigl\{\varphi;\ \iint_{\R^2\times D}\frac{\varphi^2}M\dx\dq<\infty\bigr\},\\
\hl\sigma=\bigl\{\varphi:\Omega\times D\to\C;\ \normhl\varphi\sigma:=\bigl\|\|\varphi\|_{H^{\sigma}(\Omega)}\bigr\|_{L^2_M(D)}<\infty\},\\
\hbarl\sigma=\bigl\{\varphi:\R^2\times D\to\C;\ \normhbarl\varphi\sigma:=\bigl\|\|\varphi\|_{H^{\sigma}(\R^2)}\bigr\|_{L^2_M(D)}<\infty\},\\
\hdoth\sigma=\bigl\{\varphi:\Omega\times D\to\C;\ \normhdoth\varphi\sigma:=\bigl\|\|M\nabla_q\bigl(\frac\varphi M\bigr)\|_{H^{\sigma}(\Omega)}\bigr\|_{L^2_M(D)}<\infty\},\\
\hbardoth\sigma=\bigl\{\varphi:\R^2\times D\to\C;\ \normhbardoth\varphi\sigma:=\bigl\|\|M\nabla_q\bigl(\frac\varphi M\bigr)\|_{H^{\sigma}(\R^2)}\bigr\|_{L^2_M(D)}<\infty\},\\
\hh\sigma=\hl\sigma\cap\hdoth\sigma.
\end{gather*}
The quantities (function spaces, vector fields, etc.) with a bar on top have the $x$ variable in $\R^2$.

Below, all functions $f$ and their different versions ($f^n$, $\widetilde f^n$, $g^n$, $\overline g^n$, etc.) are assumed to belong to spaces which are $H^1_M$ in the $q$ variable (which implies boundary conditions in the $q$ variable).

We have the following Poincaré type inequality. If $\varphi=\varphi(x,q)$ is such that $\int_D\varphi \dq\equiv0$, then
\begin{equation}
  \label{poincare}
\normll\varphi\leq C\normldoth\varphi,
\end{equation}
see \cite{Chu10}.

Let now denote by $\Lambda_x^\sigma$ the Fourier multiplier $\langle D\rangle^\sigma$, \textit{i.e.} the operator of multiplication in the Fourier space by $(1+|\xi|^2)^{\frac\sigma2}$. We will always apply this operator in the $x$ variable (the functions need to be defined on $\R^2$ of course). 

The divergence of a matrix is taken along rows: for $A=(a_{ij})$ we define $\dive A=(\sum_j\partial_j a_{ij})_i$. Scalar product of matrices is defined by $A:B=\sum_{i,j}a_{ij}b_{ij}$. The tensor product of two vectors $x$ and $y$ is the matrix $x\otimes y=(x_i y_j)_{ij}$.

\section{Construction of a sequence of approximate solutions}\label{as}

Clearly  $C^\infty_0 (D)$ is dense in $L_M^2$ and a positive function in  $L_M^2$ can be approached by a sequence of positive smooth functions (cut-off and convolution preserve the sign if the cut-off and convolution functions are non-negative). Moreover, we can assume that the integral of each of the approximate functions is equal to the integral of the limit function (otherwise we can normalize each of the approximate function by multiplying with an appropriate constant). In the $x$ variable, the usual standard smoothing procedure is also achieved by cut-off and convolution and these operations preserve the sign. Therefore, there exists ${f}^n_0\in C^\infty_0(\overline\Omega\times D)$, non-negative, such that 
\begin{equation}\label{fcond1}
 \f^n_0 \to\f_0 \quad\text{in }\hl s
\end{equation}
and
\begin{equation}
  \label{intfninit}
 \int_{D}{f}^n_0\dq\equiv R^{-2}. 
\end{equation}

Let us denote by $\PP $ the Leray projector, \textit{i.e.} the orthogonal projection operator from
$(L^2(\Omega))^d$ onto $H$, and by $\A$ the Stokes operator defined
by $ \A = - \PP  \Delta $. It is well known that $\A$ is a self-adjoint
operator on $H$ with compact inverse and that 
$ D (\A^{\sigma/2})= (H^\sigma(\Omega))^2 \cap V,\; \forall \sigma \in 
[1, \frac 3 2 )$ and  $ D (\A^{\sigma/2})= (H^\sigma(\Omega))^d \cap H,\; \forall \sigma \in 
[0, \frac 1 2 )$ with equivalent norms (see \cite{FM70}).  

Denote $\lambda_1,\lambda_2,\dots, \lambda_n,\dots$ the  sequence 
of eigenvalues of $\A$ and $v_1,v_2,\dots , v_n,\dots$ the
corresponding eigenvectors that form an orthonormal basis in $H$.  
Let $H_n:=\mathscr{L}\{v_1,v_2,\dots, v_n\}$ be the vector space spanned by the first $n$ eigenvectors of $\A$, and $\PP _n$ the orthogonal projection of $L^2(\Omega)$ onto $H_n$.  We endow $H_n$ with the $L^2$ norm making it a Hilbert space.  We observe that for any $\sigma\in [0,\frac12)$ there exists some constant $C(\sigma,\Omega)$ independent of $n$ such that for any $g\in H^\sigma(\Omega)$ we have 
\begin{equation}\label{boundpn}
  \|\PP_ng\|_{H^\sigma(\Omega)}\leq C(\sigma,\Omega) \|g\|_{H^\sigma(\Omega)}.
\end{equation}
Indeed, it is well-known that $\PP$ is bounded on $H^\sigma(\Omega)$. Then $\PP g\in H^\sigma(\Omega)\cap H= D (\A^{\sigma/2})$. Moreover, $\PP_n$ is an orthogonal projection in $D (\A^{\sigma/2})$ so we can write the following sequence of estimates:
\begin{multline*}
\|\PP_ng\|_{H^\sigma(\Omega)}\leq C \|\A^{\sigma/2}\PP_n\PP g\|_{L^2(\Omega)}= C\|\PP_n\A^{\sigma/2}\PP g\|_{L^2(\Omega)} \\
\leq C\|\A^{\sigma/2}\PP g\|_{L^2(\Omega)} 
\leq C' \|\PP g\|_{H^\sigma(\Omega)}
\leq C'' \|g\|_{H^\sigma(\Omega)},
\end{multline*}
which proves \eqref{boundpn}.

Letting $\PP $ operate on relation \eqref{fequ} leads to
\begin{equation*}
\dfrac{\partial u}{\partial t}+\al_1 \A u+\PP (u\cdot\nabla u)=\al_2 \PP  \left[ \nabla_x \cdot \int_D \dfrac{q\otimes q}{1-|q|^2}\f \dq  \right] 
\end{equation*}

We consider the following approximation problem: find $ (u^n, \f^n)$
with $u^n \in C([0, T], H_n)$,  such that 
\begin{equation}\label{fas5}
\partial_t u^n+\al_1 \A u^n+\PP _n 
(u^n\cdot\nabla u^n)=\al_2 \PP _n \left[ \nabla_x \cdot \int_D 
\dfrac{q\otimes q}{1-|q|^2}\f^n \dq  \right] 
\end{equation}
and
\begin{equation}\label{fas6}
\partial_t \f^n + u^n \cdot \nabla_x \f^n
-\al_3 \nabla_q\cdot \left[M\nabla_q\left(\dfrac{{\f}^n}{M}
  \right)  \right]+ \nabla_q\cdot \left(\sigma(u^n) q {\f}^n 
\right)=0
\end{equation}
with respect to the initial conditions
\begin{equation}
u^n\bigl|_{t=0}=u^n_0\equiv\PP _n u_0, \qquad
{\f}^n\bigl|_{t=0}=\f^n_0. \label{fas9} 
\end{equation}
and such that $f^n$ has $H^1_M$ regularity in the $q$ variable. We will later use that
\begin{equation}
  \label{cond2}
 u^n_0\to u_0\qquad\text{in } H^s(\Omega).
\end{equation}

\section{Global existence of the approximate solutions}
\label{approxex}

Let us first remark that for any $ f \in H_n $ and any  $ m \in 
\mathbb{N} $, one has
\begin{equation}
\label{maj-Hn}
\| f \|_{H^m(\Omega)} \leq C(\Omega,m,n) \| f \|_{L^2(\Omega)}.
\end{equation}
Indeed, if $ f = \sum_{i = 1}^n \al_i  v_i $ then
$$
\| f \|_{H^m(\Omega)} \leq C \| \A^{m/2} f \|_{L^2(\Omega)}  = C \left (   \sum_{i = 1}^n \al_i^2 \lambda_i^m \right )^{1/2}
\leq C \max_{i \in {1, \dots n}} \lambda_i^{m/2} \|f\|_{L^2(\Omega)}.
   $$

Throughout this section $C$ denotes a constant that depends on $n$, $\Omega$, material coefficients and other constants. It may change from one line to another.

The existence of the approximate solutions is granted by the following theorem:
\begin{theorem}\label{existapprox}
There exists a global solution $ (u^n , \f^n) $  to the problem 
\eqref{fas5}-\eqref{fas9}, such that $ u^n \in C^0(\R_+; H_n) $.
\end{theorem}
\begin{proof}
In this proof the various constants $C, \; C_1, \; C_2,  \dots $ may depend on $n$ but are independent of time. We fix an arbitrary finite time $T>0$ and show that we can solve \eqref{fas5}-\eqref{fas9} up to time $T$ such that  $ u^n \in C^0([0,T]; H_n) $.

Suppose that $u^n$ is an element of $  C^0([0, T] ; H_n) $. From \eqref{maj-Hn} we deduce that
$$
\sup_{t \in [0, T]} \| u^n \|_{W^{1, \infty}(\Omega)} < + \infty.
  $$ 
Therefore one can construct  the flow $\chi_n(t, y)$ of $u^n$ as the unique solution to the equation
$$
\partial_t \chi_n(t, y) = u^n(t, \chi_n(t, y)), \quad \chi_n(0, y) = y.
  $$
Since the Jacobian determinant of $\chi_n(t, \cdot)$ is equal
to  1 we deduce that for any 
$ t \in [0, T] , \; 
\chi_n(t, \cdot) $ is a $C^\infty$ - diffeomorphism from $\Omega$ to $\Omega$.

Let $\ft^n(t,y,q)=\f^n(t,\chi_n(t,y),q)$.  Clearly $\f^n$ solves 
\eqref{fas6}--\eqref{fas9} if and only if $\ft^n$ solves
\begin{equation}
  \label{feqpsib}
\begin{cases}
\partial_t \ft^n+\nabla_q \cdot [\sigma(u^n)\circ\chi_n \,
q\ft^n] -
\al_3 \nabla_q \cdot \left [M\nabla_q\left (\frac{\ft^n} M\right)
\right ] = 0
\\
\qquad \ft^n(0,y,q)= \f^n_0(y,q).
\end{cases}
\end{equation}

In the equation above, the variable $y$ plays the role of a parameter
only. The existence, uniqueness and smoothness of a solution
 $\ft^n $ to
\eqref{feqpsib} which is $H^1_M$ in $q$ was proved in \cite{Mas08a}.
This allows to construct $\ft^n$, and therefore $\f^n$, if $u^n$ 
is given. We denote by $\S$ the operator that gives $\f^n$ in terms of $u^n$, $\f^n=\S(u^n)$. 

Since $u^n$ is smooth enough w.r.t. $x$ we also deduce that  $\ft^n$ is
smooth  enough w.r.t. $y$ and the same holds true for $\f^n$. Then $\f^n$
satisfies \eqref{fas6} and \eqref{fas9}. We observe moreover that
\begin{equation}\label{intfn}
\int_D f^n\dq\equiv\frac1{R^2}.  
\end{equation}
Indeed, from \eqref{intfninit} we know that the above relation is satisfied at time $t=0$. If we integrate with respect to $q$ relation \eqref{fas6} we have that the quantity $\int_D f^n\dq$ is transported by the vector field $u^n$ so it must be constant.

We conclude from the preceding observations that it suffices to show that there exists a global solution 
$u^n\in C^0([0, T]; H_n)$ of the following equation
\begin{equation}
  \label{equn}
\partial_t u^n+\PP _n(u^n\cdot \nabla  u^n) + \al_1 
\A u^n=\al_2 \PP _n \left (\nabla_x \cdot \int_D\S(u^n)F(q)\dq\right), \qquad 
u^n(0,x) = \PP _n  u_0  
\end{equation}
where $F(q)$ is the following matrix:
\begin{equation*}
F(q)=  \frac{q\otimes q}{1-|q|^2}.
\end{equation*}
Indeed, if such an $u^n$ is obtained, then the couple $(u^n,\S(u^n))$ solves \eqref{fas5}--\eqref{fas9}.

To solve \eqref{equn} we will use a fixed point method. More precisely, we write \eqref{equn} under the following equivalent integral form:
\begin{equation}\label{eqint}
u^n(t)
=e^{- \al_1 t \A } \PP _n u_0 + \int_0^t e^{\al_1 (s-t) \A } \PP _n 
\bigl[-u^n\cdot \nabla u^n + \al_2 \nabla_x \cdot 
\int_D\S(u^n)F(q)\dq\bigr](s) \ds
\end{equation}
We search for $u^n$ as a fixed point of the operator
\begin{equation*}
\B:C^0([0,T_0];H_n)\to C^0([0,T_0];H_n), \qquad  \B(v) \;
\text { is given by the rhs of \eqref{eqint}}
\end{equation*}
where the time $T_0$ is to be chosen to ensure $\B$ is a contraction
mapping. Recall that $H_n$ is endowed  with the $L^2$ topology. One has:
\begin{multline}\label{bound2}
\|\B(v)-\B(v') \|_{L^2(\Omega)}
\leq \left \|{\int_0^t e^{\al_1(s-t)\A }\PP _n(v \cdot \nabla v-v' \cdot \nabla v')} \right \|_{L^2(\Omega)}\\
+ \al_2 \left \| {\int_0^t e^{\al_1(s-t)\A }\PP _n\bigl[\nabla_x \cdot 
\int_D(\S(v) - \S(v'))F(q) \dq\bigr]} \right \|_{L^2(\Omega)}\
\equiv I_1+I_2.
\end{multline}
Next:
\begin{equation}\label{bound3}
\begin{split}
I_1&=\Bigl\| {\int_0^t e^{\al_1 (s-t)\A } \PP _n[(v-v')\cgrad v+v'\cgrad(v- v')]} \Bigr\|_{L^2(\Omega)}\\
&\leq \int_0^t \left \| {(v-v')\cgrad v+v'\cgrad(v- v')} \right \|_{L^2(\Omega)}\\ 
&\leq C\int_0^t \|v-v'\|_{H^2(\Omega)}(\|v\|_{H^2(\Omega)}+\|v'\|_{H^2(\Omega)})  \\
&\leq Ct\sup_{[0,t]}\|{v-v'} \|_{L^2(\Omega)} (\sup_{[0,t]}\|v\|_{L^2(\Omega)} +\sup_{[0,t]}\|v'\|_{L^2(\Omega)} )
\end{split}
\end{equation}
where we used relation \eqref{maj-Hn}. To bound $I_2$, we observe first that we have the following inequality
\begin{equation*}
\| \PP _n \nabla \cdot  h\|_{L^2(\Omega)}^2=\sum_{i=0}^n| \scal{\nabla \cdot  h}
{\phi_i}|^2  
=\sum_{i=0}^n|  \scal {h}{\nabla\phi_i}|^2  
\leq\sum_{i=0}^n \|h\|_{L^2(\Omega)}^2\| {\nabla\phi_i}\|_{L^2(\Omega)}^2  
\leq C(n)\| h \|_{L^2(\Omega)}^2.
\end{equation*}
Therefore
\begin{equation}\label{bound4}
\begin{split}
I_2&\leq \al_2 \int_0^t \left \| {\PP _n \bigl[\nabla_x \cdot 
\int_D(\S(v)-\S(v'))F(q)
\dq\bigr]} \right \|_{L^2(\Omega)}  \\
&\leq Ct\sup_{[0,t]}\left \| {\int_D(\S(v)-\S(v'))F(q)\dq}
\right \|_{L^2(\Omega)} \\
&\leq Ct\sup_{[0,t]}\left \|\frac{\S(v)-\S(v')}{\sqrt M}
\right \|_{L^2(\Omega\times D)}\|\sqrt MF\|_{L^2(D)}\\
&\leq Ct\sup_{[0,t]}\|\S(v)-\S(v')\|_{\ll},
\end{split}
\end{equation}
where we used that $\delta>1$ to have $\|\sqrt MF\|_{L^2(D)}<\infty$.
Next, we remark that $\Phi\equiv\S(v)-\S(v')$ solves the equation
\begin{multline*}
\partial_t \Phi+(v-v')\cdot \nabla_x\S(v)+v'\cdot \nabla_x\Phi+
\nabla_q \cdot [\sigma(v-v')q \S(v)]\\
+\nabla_q \cdot (\sigma(v')q\Phi) 
- \al_3 \nabla_q \cdot \left [M\nabla_q\left (\frac{\Phi} M
\right )\right ]=0, 
\qquad \Phi(0,x,q)=0.  
\end{multline*}
We multiply the above relationship by $\dfrac\Phi M$ and integrate in $x$
and $q$ to obtain, after some straightforward calculations that
\begin{equation*}
\begin{aligned}
\ddt \| \Phi \|_{\ll}^2
+2 \al_3 &\|\Phi\|_{\ldoth }^2=
2\iint  (v'-v)\cdot \nabla_x\S(v) 
\frac \Phi M  
+2\iint \sigma(v-v'):\left [\S(v)q\otimes \nabla_q\left(\frac\Phi
M\right )\right ] \\
& \qquad  \qquad   \qquad + 2\iint \sigma(v'):\Bigl[ \Phi q\otimes \nabla_q \Bigl ( 
\frac \Phi M \Bigr )\Bigr]  \\
\leq &2\|{v-v'}\|_{L^\infty(\Omega)} \|{\nabla_x\S(v)}\|_{\ll}\|\Phi\|_{\ll} 
+2 \|{\sigma(v-v')}\|_{L^\infty(\Omega)}
\|\S(v)\|_{\ll}\|\Phi\|_{\ldoth } \\
&\qquad\qquad\qquad+ 2\|{\sigma(v')}\|_{L^\infty(\Omega)} \|\Phi\|_{\ll} \|\Phi\|_{\ldoth }.
\end{aligned}
\end{equation*}
As in \eqref{intfn} we have that $\int_D\Phi\dq\equiv0$ so the Poincaré inequality \eqref{poincare} holds true for $\Phi$. We deduce, using inequality \eqref{maj-Hn}, that:
\begin{multline}
\label{bound5}
\ddt \| \Phi \|_{\ll}^2
+ \al_3 \|\Phi\|_{\ldoth }^2
\leq C \|v - v'\|_{L^2(\Omega)}^2 \bigl [  \|{\nabla_x\S(v)}\|_{\ll}^2 
+ \|\S(v)\|_{\ll}^2 \bigr ]\\
 + C \|v'\|_{L^2(\Omega)}^2  \| \Phi \|_{\ll}^2.
\end{multline}

Recall that $\S(v)$ solves the equation
\begin{equation}
\label{eqv}
\partial_t \S(v)+v\cdot \nabla_x \S(v)+\nabla_q \cdot
[\sigma(v)q\S(v)]
-\al_3 \nabla_q \cdot \left [M\nabla_q\left (\dfrac{\S(v)} M\right)
\right ]= 0,\quad 
\S(v)(0,x,q)= \f^n_0(x,q)
\end{equation}

Multiplying the equation of $\S(v)$  by $\dfrac {\S(v)} M$ and integrating w.r.t. $x,q$, we obtain
after similar estimates the following relation:
\begin{equation}
\label{bound5p}
\ddt  \| \S(v) \|_{\ll}^2
+ \al_3 \|\S(v)\|_{\ldoth }^2
 \leq C \|v\|_{L^2(\Omega)}^2  \| \S(v) \|_{\ll}^2  .
\end{equation}

We will use in what follows several times the following simplified version of the Gronwall
inequality:
$$
y' \leq a_1 y + a_2, \qquad y(0) = y_0
  $$
with $a_1, a_2$ positive constants, implies
$$
y(t) \leq (y_0 + a_2 t) e^{a_1 t}.
  $$

From \eqref{bound5p} it follows:
\begin{equation}
\label{bound6}
\|\S(v)(t)\|_{\ll}^2 \leq R_1^2 \left ( \sup_{t \in [0, T]} \|v\|_{L^2(\Omega)}^2 
\right )
\end{equation}
where we denoted, for any $z \in \mathbb{R}$:
$$
R_1^2(z) = \|\f^n_0\|_{\ll}^2 \exp{ (C zT) }.
  $$
Similarly, differentiating \eqref{eqv} with respect to $x_i$, 
multiplying by $\dfrac{\partial_{x_i}\S(v)}M$ and summing over $i$, 
gives after integration and some straightforward estimates that
\begin{align*}
\ddt   \| \nabla_x \S(v)& \|_{\ll}^2
+ 2\al_3 \|\nabla_x \S(v)\|_{\ldoth }^2\\
\leq  &C\|\nabla v\|_{L^\infty(\Omega)}\normll{\nabla_x \S(v)}^2
+C\|\sigma(v)\|_{L^\infty(\Omega)}\normll{\nabla_x \S(v)}\normldoth{\nabla_x \S(v)}\\
&\hskip 6cm +C\|\nabla\sigma(v)\|_{L^\infty(\Omega)}\normll{\S(v)}\normldoth{\nabla_x \S(v)}\\
\leq  &\al_3 \|\nabla_x \S(v)\|_{\ldoth }^2 
+C\normll{\nabla_x \S(v)}^2\|v\|_{L^2(\Omega)}(1+\|v\|_{L^2(\Omega)})
+C\normll{\S(v)}^2\|v\|_{L^2(\Omega)}^2
\end{align*}
Using \eqref{bound6} and the Gronwall lemma implies that
\begin{equation}
\label{bound7}
\|\nabla_x \S(v)(t)\|_{\ll}^2 \leq R_2^2 \left ( \sup_{t \in [0, T]} \|v\|_{L^2(\Omega)}^2 
\right )
\end{equation}
where
$$
R_2^2(z) = \left [ \| \nabla_x\f^n_0\|_{\ll}^2 + C z 
R_1^2(z) T \right ] \exp{ (C (z+ 1) T) }.
  $$
From \eqref{bound5}, \eqref{bound6} and  \eqref{bound7} we get
\begin{equation}
\nonumber
\| \S(v) - \S(v') \|_{\ll}^2 \leq C T
\| v - v'\|_{W_T}^2 \left [ R_1^2
\left (\|v\|_{W_T}^2 \right ) +
R_2^2 \left (\|v\|_{W_T}^2 \right )  \right ] 
\exp{\left (C T\|v'\|_{W_T}^2 \right )}
\end{equation}
where we denoted 
$$W_T = C([0, T]; H_n).$$

Combining the above with \eqref{bound2}, \eqref{bound3} and 
\eqref{bound4} yields
\begin{multline}\label{contract}
\|\B(v)-\B(v')\|_{W_T}\leq C_1T\|v-v'\|_{W_T} \left ( \|v\|_{W_T}+
\|v'\|_{W_T} \right ) \\
+  C_1 T^{3/2}\|v-v'\|_{W_T} \left [ R_1(\|v\|_{W_T}^2) + 
R_2(\|v\|_{W_T}^2)  \right ]
\exp{(C_1  \|v'\|_{W_T}^2 T)}
\end{multline}
By a procedure similar in nature to the one detailed right above, 
one gets:
\begin{equation}
\label{bounded}
\|\B(v)\|_{W_T}\leq \|u^n_0\|_{L^2(\Omega)} + C_2 T \|v\|_{W_T}^2 
+ C_2 T R_1(\|v\|_{W_T}^2)
\end{equation}
Let now $K_0$ be such that
\begin{equation*}
\max\left(\|u^n_0\|_{L^2(\Omega)}, \|\f^n_0\|_{\ll},
  \|\nabla_x \f^n_0 \|_{\ll} 
\right)\leq K_0. 
\end{equation*}
Now taking a fixed $s_0$ such that  $ s_0 >  K_0 $
(for example $s_0 = K_0 + 1$) we can choose a $T_0$ sufficiently 
small such that
\begin{equation*}
K_0+C_2T_0 s_0^2+C_2T_0 R_1(s_0^2) \leq s_0
\end{equation*}
and
\begin{equation*}
2 C_1 T_0 s_0 + C_1 T_0^{3/2} \bigl[R_1(s_0^2) + R_2(s_0^2) \bigr]
\exp(C_1 s_0^2 T_0)<1  
\end{equation*}
The above assumptions together with \eqref{contract} and
\eqref{bounded} show that the operator $\B$ is a contraction from the
closed ball $\overline{B}(0,s_0)$ of $W_{T_0}$ onto itself. The fixed
point theorem can therefore be applied to grant the existence of a local
in time solution on $[0,T_0]$. Moreover, the local time existence
$T_0$ depends only on the bound $K_0$ for the initial data $ u^n_0$ and
$\S(v)(0)=\f^n_0$. Starting from time $T_0$, the same argument can
be applied to extend the solution, and so on. We justify now that $T$ 
can be reached in this way in a finite number of
steps. When re-applying the fixed point
argument from time $T_0$, the new time of existence depends only on
$\max{(\|u^n(T_0)\|_{L^2(\Omega)}, \|\f^n(T_0)\|_{\hl1 })}$. 
But the estimates shown in the next section
imply that any solution $(u^n(t),\f^n(t))$ on a time interval included in $[0,T]$ can be bounded in $L^2\times \hl1 $ independently of $t\in[0,T]$. Indeed, we will show in particular some $H^s(\Omega)$ bounds on $u^n$. This implies $L^2(\Omega)$ bounds on $u^n$ and by estimates \eqref{bound6} and \eqref{bound7} some $\hl1$ bounds on $\f^n$. This means that the time-existence $T_0$ can be chosen the same at each step, so the time $T$ will be reached in a finite number of steps. This completes the proof of Theorem \ref{existapprox}.
\end{proof}

\section{Uniform estimates for the sequence of approximate solutions and end of the proof}
\label{unif}

Let us introduce the new function
\begin{equation*}
\psi^n=f^n-aM(q) 
\end{equation*}
where
\begin{equation}
  \label{valoarea}
 a=\frac1{R^2\int_DM(q)\dq}=\frac{\delta+1}{\pi R^2}. 
\end{equation}

It is not hard to check that the couple $(u^n,f^n)$ verifies \eqref{fas5}--\eqref{fas9} if and only if the couple $(u^n,\psi^n)$ verifies the system of equations \eqref{as5}--\eqref{as9} below:
\begin{equation}\label{as5}
\partial_t u^n+\al_1 \A u^n+\PP _n 
(u^n\cdot\nabla u^n)=\al_2 \PP _n \left[ \nabla_x \cdot \int_D 
\dfrac{q\otimes q}{1-|q|^2}\psi^n \dq  \right] 
\end{equation}
and
\begin{equation}\label{as6}
\partial_t \psi^n + u^n \cdot \nabla_x \psi^n
-\al_3 \nabla_q\cdot \left[M\nabla_q\left(\dfrac{{\psi}^n}{M}
  \right)  \right]+ \nabla_q\cdot \left(\sigma(u^n) q {\psi}^n 
\right)=-a\nabla_q\cdot[\sigma(u^n)qM]
\end{equation}
with respect to the initial conditions
\begin{equation}
u^n\bigl|_{t=0}=u^n_0, \qquad 
{\psi}^n\bigl|_{t=0}=\psi^n_0\equiv f^n_0-aM(q). \label{as9} 
\end{equation}

We observe that relation \eqref{fcond1} can be rewritten as
\begin{equation}\label{cond1}
 \psi^n_0 \to\psi_0\equiv\f_0-aM \quad\text{in }\hl s
\end{equation}
as $n\to\infty$, while relation \eqref{intfn} is equivalent to
\begin{equation}\label{intpsin}
\int_D\psi^n\dq\equiv0.  
\end{equation}

\begin{remark}
We observe that
\begin{equation*}
\nabla_q\cdot(\sigma(u^n)qM)=\nabla_q\cdot(\sigma(u^n)q)M+(\sigma(u^n)q)\cdot\nabla_q M=(\sigma(u^n)q)\cdot\nabla_q M.
\end{equation*}
We used above that $tr[\sigma(u^n)]=0$ so $\nabla_q\cdot(\sigma(u^n)q)=0$. Since $\nabla_q M$ is proportional to $q$ we infer that if $\sigma(u^n)$ is skew-symmetric then 
\begin{equation*}
\nabla_q\cdot(\sigma(u^n)qM)=0.
\end{equation*}
We conclude that in the corotational case the right-hand side of the equation of $\psi^n$ given in relation \eqref{as6} vanishes.
\end{remark}

We will show now some uniform (in $n$) estimates on the approximate solutions  $u^n$ and $\psi^n$ constructed above. In the calculations below, $C$ is a generic notation for a constant that does not depend on $n$ and its numerical value changes from one calculation to another.

\paragraph{\large $L^2$ estimates in the corotational case.} Our first bound is a $L^2$ energy estimate on $u^n$. 
\begin{lemma} \label{energbound}
We have that
\begin{equation*}
\|u^n(t)\|_{L^2(\Omega)}^2+\al_1\int_0^t\|\nabla u^n\|_{L^2(\Omega)}^2\leq \|u_0\|_{L^2(\Omega)}^2+\al_4\normll{\psi_0}^2\qquad \forall t\geq0,  
\end{equation*}
where
\begin{equation*}
\al_4=\frac{\pi\al_2^2}{8\delta^4\al_1\al_3}.  
\end{equation*}
\end{lemma}
\begin{proof}
This estimate is well-known in the case of domains without boundaries. We only have to check that it goes through in the case of domains with boundaries, and also that it is compatible with our approximation procedure. Moreover, we need to compute  the precise best constants so we will give a detailed proof.

Recall that  $\sigma(u^n)=\nabla u^n-(\nabla u^n)^t$. Given that $\sigma(u^n)$ is trace-free we have that
\begin{equation*}
\nabla_q\cdot(\sigma(u^n)q\psi^n)=\sigma(u^n)^t:q\otimes\nabla_q\psi^n.
\end{equation*}
We multiply next \eqref{as6} by $\psi^n/M$ and integrate in $x$ and $q$. After recalling that the right-hand side of \eqref{as6} vanishes in the corotational case, we obtain
\begin{equation}
\begin{split}\label{enfn}
\frac12\ddt\normll{\psi^n}^2+\al_3\normldoth{\psi^n}^2
&=-\iint_{\Omega\times D}\sigma(u^n)^t:q\otimes\nabla_q\psi^n \frac{\psi^n}M\\
&=-\frac12 \iint_{\Omega\times D}\sigma(u^n)^t:q\otimes\nabla_q(|\psi^n|^2)M\\
&=\frac12 \iint_{\Omega\times D}\sigma(u^n)^t:q\otimes\nabla_qM |\psi^n|^2\\
&=0.
\end{split}
\end{equation}
since $\sigma(u^n)$ is trace free and skew-symmetric and $q\otimes\nabla_qM$ is symmetric.

We multiply now \eqref{as5} by $u^n$ and integrate in space to obtain
\begin{equation}\label{enun}
\begin{split}
\frac12\ddt\|u^n\|_{L^2(\Omega)}^2+\al_1\|\nabla u^n\|_{L^2(\Omega)}^2
&=\al_2 \int_\Omega u^n \nabla_x \cdot \int_D 
\dfrac{q\otimes q}{1-|q|^2}\psi^n \dq \dx\\
&=-\al_2\iint_{\Omega\times D}\nabla u^n: \dfrac{q\otimes q}{1-|q|^2}\psi^n 
\end{split}
\end{equation}
Next, 
\begin{multline*}
-\al_2\iint_{\Omega\times D}\nabla u^n: \dfrac{q\otimes q}{1-|q|^2}\psi^n
= \frac{\al_2}{2\delta}\iint_{\Omega\times D}\nabla u^n: q\otimes \nabla_q M \ \frac{\psi^n}M 
=-\frac{\al_2}{2\delta}\iint_{\Omega\times D}\nabla u^n: q\otimes \nabla_q \Bigl(\frac{\psi^n}M\Bigr)M\\
\leq  \frac{\al_2}{2\delta}\|\nabla u^n\|_{L^2(\Omega)}\|q\sqrt M\|_{L^2(D)}\normldoth{\psi^n}
\leq \frac{\al_1}2\|\nabla u^n\|_{L^2(\Omega)}^2+\frac{\pi\al_2^2}{8\delta^4\al_1} \normldoth{\psi^n}^2.
\end{multline*}
We used above that
\begin{equation}\label{boundel}
\|q\sqrt M\|_{L^2(D)}=  \frac{\sqrt\pi}{\sqrt{(\delta+1)(\delta+2)}}\leq\frac{\sqrt\pi}{\delta}.
\end{equation}
We infer now from \eqref{enun} that
\begin{equation*}
 \ddt\|u^n\|_{L^2(\Omega)}^2+\al_1\|\nabla u^n\|_{L^2(\Omega)}^2 
\leq \frac{\pi\al_2^2}{4\delta^4\al_1} \normldoth{\psi^n}^2.
\end{equation*}
Integrating the above relation and \eqref{enfn} in time implies that
\begin{equation*}
\|u^n(t)\|_{L^2(\Omega)}^2+\al_1\int_0^t\|\nabla u^n\|_{L^2(\Omega)}^2\leq \|u^n_0\|_{L^2(\Omega)}^2+\  \frac{\pi\al_2^2}{8\delta^4\al_1\al_3}\normll{\psi^n_0}^2.  
\end{equation*}
This completes the proof of the lemma.
\end{proof}
\begin{remark}
A similar $L^2$ energy estimate holds true in the general case too, see \cite{JBLO06}.   
\end{remark}

\paragraph{\large $H^s$ estimates for $u^n$.} The next step is to prove $H^s$ estimates for $u^n$.
%
\begin{lemma}
There exist a constant $C_1> 0$ depending only on $\Omega$ and $s$ such that:
\begin{multline}\label{uhs}
\ddt \| \A^{\frac s2}u^n  \|^2_{L^2(\Omega)}
+  \al_1 \| \A^{\frac{1+s}2}u^n \|^2_{L^2(\Omega)} \leq
C_1\min\bigl[h(t) \| \A^{\frac s2}u^n\|_{L^2(\Omega)}^2,  \| \A^{\frac s2}u^n\|_{L^2(\Omega)}\| \A^{\frac {s+1}2}u^n\|_{L^2(\Omega)}^2 \bigr]\\
 +  C_1\frac{\al_2}{\delta^2} \| \A^{\frac{1+s}2}u^n \|_{L^2(\Omega)}  \normhdoth{\psi^n}s
\end{multline}
where the function $h$ is integrable on $\R_+$ and satisfies
\begin{equation}\label{inegh}
\int_0^\infty h(t) \, dt \leq 
\al_1^{-\frac4s}\Bigl( \|u_0 \|_{L^2(\Omega)}^2 + \al_4 \normll{\psi_0}^2\Bigr)^{\frac2s}.
\end{equation} 
\end{lemma}
\begin{proof}
Let us take the scalar product in $L^2(\Omega)$ of  \eqref{as5} 
with $\A^{s} u^n$. We get
\begin{multline*}
\dfrac{1}{2}\ddt \| \A^{\frac s2}u^n  \|^2_{L^2(\Omega)}
+  \al_1 \| \A^{\frac{1+s}2}u^n  \|^2_{L^2(\Omega)} =
-\int_\Omega \A^{\frac{s-1}2} \PP _n (u^n\cdot \nabla u^n)  \cdot
\A^{\frac{s+1}2} u^n \dx \\
 +  \al_2 \int_\Omega \A^{\frac{s-1}2} \PP _n  \bigl [ \nabla_x \cdot \int_D 
\dfrac{q\otimes q}{1-|q|^2} \psi^n \dq  \bigr]  \cdot \A^{\frac{s+1}2} u^n \dx.
\end{multline*}
Using that $s\in (1,\frac32)$ and relation \eqref{boundpn} we infer that
\begin{multline}\label{ee1}
\ddt \| \A^{\frac s2}u^n  \|^2_{L^2(\Omega)}
+  2\al_1 \| \A^{\frac{1+s}2}u^n  \|^2_{L^2(\Omega)}  \leq
C \| u^n\cdot \nabla u^n\|_{H^{s-1}} \| \A^{\frac{1+s}2}u^n  \|_{L^2(\Omega)}  \\
 +  \al_2  \bigl\| \A^{\frac{s-1}2} \PP _n  \bigl [ \nabla_x \cdot \int_D 
\dfrac{q\otimes q}{1-|q|^2} \psi^n \dq  \bigr]  \bigr\|_{L^2(\Omega)}\| \A^{\frac{1+s}2}u^n  \|_{L^2(\Omega)}.  
\end{multline}

Standard product rules and interpolation in Sobolev spaces imply that
\begin{equation*}
 \| u^n\cdot \nabla u^n\|_{H^{s-1}}  \leq C \| u^n \|_{H^{\frac s2}}
\| u^n \|_{H^{1+\frac s 2}}
\leq C \| u^n \|_{L^2}^{1-\frac s2}\| u^n \|^{\frac s2}_{H^1}
\| u^n \|^{\frac s2}_{H^{s}}\| u^n \|^{1-\frac s2}_{H^{1+s}}. 
\end{equation*}
We infer that
\begin{equation*}
\begin{split}
 C \| u^n\cdot \nabla u^n\|_{H^{s-1}} \| \A^{\frac{1+s}2}u^n  \|_{L^2(\Omega)}  
&\leq C \| u^n \|_{L^2}^{1-\frac s2}\| u^n \|^{\frac s2}_{H^1}
\| \A^{\frac s2}u^n \|^{\frac s2}_{L^2}\| \A^{\frac{1+s}2}u^n \|^{2-\frac s2}_{L^2}\\
&\leq \al_1\| \A^{\frac{1+s}2}u^n \|^2_{L^2} +C \al_1^{1-\frac4s}\| u^n \|_{L^2}^{\frac 4s-2}\| \nabla u^n \|^2_{L^2}
\| \A^{\frac s2}u^n \|^2_{L^2}.
\end{split}
\end{equation*}
On the other hand, using again the product rules in Sobolev spaces we can also estimate
\begin{equation*}
\| u^n\cdot \nabla u^n\|_{H^{s-1}}\leq C\| u^n\|_{H^{s}}\| \nabla u^n\|_{H^{s-1}}\leq C\| u^n\|^2_{H^{s}}
\leq C\| \A^{\frac s2}u^n \|^2_{L^2}\leq C\| \A^{\frac s2}u^n \|_{L^2}\| \A^{\frac {s+1}2}u^n \|_{L^2}.
\end{equation*}
Hence
\begin{multline}
  \label{ineg1}
C \| u^n\cdot \nabla u^n\|_{H^{s-1}} \| \A^{\frac{1+s}2}u^n  \|_{L^2(\Omega)}  
\leq   \al_1\| \A^{\frac{1+s}2}u^n \|^2_{L^2}\\ 
+C\min\bigl[\al_1^{1-\frac4s}\| u^n \|_{L^2}^{\frac 4s-2}\| \nabla u^n \|^2_{L^2}
\| \A^{\frac s2}u^n \|^2_{L^2}, \| \A^{\frac s2}u^n \|_{L^2}\| \A^{\frac {s+1}2}u^n \|^2_{L^2}\bigr]
\end{multline}

To bound the last term in \eqref{ee1} we observe that
\begin{align*}
\nabla_x \cdot \int_D 
\dfrac{q\otimes q}{1-|q|^2} \psi^n \dq
&=-\frac1{2\delta}\nabla_x \cdot\int_D\frac{\psi^n}M[\nabla_q\otimes(qM)]\dq
+\frac1{2\delta}\nabla_x \cdot\bigl[\operatorname{Id}\int_D  \frac{\psi^n}M\dq \bigr]\\
&=\frac1{2\delta}\nabla_x \cdot\int_D\nabla_q\bigl(\frac{\psi^n}M\bigr)\otimes qM\dq
+\frac1{2\delta}\nabla_x \cdot\bigl[\operatorname{Id}\int_D  \frac{\psi^n}M\dq \bigr]
\end{align*}
where $\operatorname{Id}$ denotes the identity matrix. The last term above is a gradient, so it belongs to the kernel of $\PP_n$. Therefore, the last term in \eqref{ee1} may be estimated as follows
\begin{align}
\bigl\| \A^{\frac{s-1}2} \PP _n  \bigl [ \nabla_x \cdot \int_D 
\dfrac{q\otimes q}{1-|q|^2} \psi^n \dq  \bigr]  \bigr\|_{L^2(\Omega)}
&= \frac1{2\delta} \bigl\| \A^{\frac{s-1}2} \PP _n  \bigl [ \nabla_x \cdot \int_D\nabla_q\bigl(\frac{\psi^n}M\bigr)\otimes qM\dq\bigr]  \bigr\|_{L^2(\Omega)}\notag\\
&\leq \frac C\delta\normhdoth{\psi^n}s\|q\sqrt M\|_{L^2(D)}\label{in-tau}\\
&\leq  \frac C{\delta^2}\normhdoth{\psi^n}s\notag
\end{align}
where we used \eqref{boundel}.

Relation \eqref{uhs} follows from relations \eqref{ee1}, \eqref{ineg1} and \eqref{in-tau} if we set
\begin{equation*}
h(t)=  C \al_1^{1-\frac4s}\| u^n \|_{L^2}^{\frac 4s-2}\| \nabla u^n \|^2_{L^2}
\end{equation*}
for some suitable constant $C$. Relation \eqref{inegh} is a consequence of Lemma \ref{energbound} and this completes the proof of the lemma.
\end{proof}

\paragraph{\large $H^s$ estimates for $\psi^n$.} We  need now estimates on $\psi^n$. For technical reasons, in order to obtain these estimates we need to work in $\R^2$ for the $x$ variable. Because there are no boundary conditions for $\psi$ in the $x$ variable, it is possible to extend the equation  of $\psi^n$ to $\R^2\times D$. 

Let $E$ be a total extension operator from $\Omega$ to $\R^2$, \textit{i.e.} a linear operator bounded from $H^\sigma(\Omega)$ to $H^\sigma(\R^2)$ for every $\sigma\geq0$ (in fact we  only need it to be bounded for $\sigma\in[0,3]$, that is we only need a 3-extension operator). The existence of such operators is well-known, see\textit{ e.g.} \cite[Chapter 4]{Ada75}. 

We define now
\begin{equation*}
 \psib^n_0=E(\psi^n_0) 
\end{equation*}
so that $\psib^n_0\in \hbarl s$ and 
\begin{equation}\label{cond3}
\normhbarl{\psib^n_0}s\leq C_0  \normhl{\psi^n_0}s,
\end{equation}
where $C_0$ depends only on $\Omega$. The way the extension operator is constructed in \cite{Ada75} ensures that the integral in the $q$ variable is preserved:
\begin{equation}\label{intpsibinit}
\int_D\psib^n_0\dq\equiv0
\end{equation}
since \eqref{intpsin} holds true.

Next, we want to extend $u^n$ to a  smooth  divergence free vector field defined on $\R^2$. In order to preserve the divergence free condition, we need to introduce the stream function. For a divergence free vector field $v$ defined on $\Omega$ and vanishing on the boundary of $\Omega$, it is well-known that there exists a stream function,  \textit{i.e.} a scalar function $\str$ such that $v=\nabla^\perp\str$. Moreover, since $v$ vanishes  on $\partial\Omega$ we have that $\str$ is constant on each connected component of $\partial\Omega$. Let $\Gamma_0$ be such a connected component. Since $\Omega$ is connected, we clearly have existence and uniqueness of $\str$ if we impose that $\str$ vanishes on $\Gamma_0$. In the sequel,  we define $\str(v)$ as the unique stream function  of $v$ vanishing on $\Gamma_0$. By the Poincaré inequality, we have that $\str$  is bounded from $H^\sigma(\Omega)\cap H_0^1(\Omega)$ to $H^{\sigma+1}(\Omega)$ for all $\sigma\geq1$. 

We define now $\ub^n=\nabla^\perp  E(\str(u^n))$. Clearly $\ub^n\bigl|_\Omega=u^n$ and 
\begin{equation}\label{boundubar}
\|\ub^n\|_{H^\sigma(\R^2)}\leq C   \|E(\str(u^n))\|_{H^{\sigma+1}(\R^2)}\leq C   \|\str(u^n)\|_{H^{\sigma+1}(\Omega)}\leq C   \|u^n\|_{H^\sigma(\Omega)}\qquad \forall \sigma\geq1,
\end{equation}
where $C=C(\Omega,\sigma)$.

We finally define $\psib^n$ as the unique solution in $\R^2\times D$ of the PDE
\begin{equation}\label{eqpsib1}
\partial_t \psib^n + \ub^n \cdot \nabla_x \psib^n
-\al_3 \nabla_q\cdot \left[M\nabla_q\left(\dfrac{{\psib}^n}{M}
  \right)  \right]+ \nabla_q\cdot \left(\sigma(\ub^n) q {\psib}^n 
\right)=- a \nabla_q \cdot \left[ \sigma(\ub^n) q M \right].
\end{equation}
The existence and uniqueness of such a $\psib^n$ follows from the argument given at the beginning of the proof of Theorem \ref{existapprox} (the variable $x$ plays the role of a parameter only). By uniqueness of solutions of \eqref{as6} and \eqref{as9}, we have that $\psib^n\bigl|_{\Omega\times D}=\psi^n$. Moreover, given \eqref{intpsibinit} we can prove as for $\psi^n$ that relation \eqref{intpsin} holds true for $\psib^n$:
\begin{equation}\label{intpsibn}
\int_D\psib^n\dq\equiv0.  
\end{equation}

The following lemma gives our estimates on $\psib^n$. 

\begin{lemma}
There exist a constant $C_2> 0$ depending only on $\Omega$ and $s$ such that 
\begin{multline}\label{psihsgen}
\ddt \normhbarl{\psib^n}s^2+2\al_3\normhbardoth{\psib^n}s^2\leq C_2 \|\A^{\frac{1+s}2}u^n\|_{L^2(\Omega)}  \normhbarl{\psib^n}s\normhbardoth{\psib^n}s\\
+\frac{C_2}{R^2}\|\A^{\frac{1+s}2}u^n\|_{L^2(\Omega)}  \normhbardoth{\psib^n}s
\end{multline}
in the general case and 
\begin{equation}\label{psihscor}
\ddt \normhbarl{\psib^n}s^2+2\al_3\normhbardoth{\psib^n}s^2\leq C_2 \|\A^{\frac{1+s}2}u^n\|_{L^2(\Omega)}  \normhbarl{\psib^n}s\normhbardoth{\psib^n}s.
\end{equation}
in the corotational case.
\end{lemma}
\begin{proof}
We apply the operator $\Lambda_x^s$ to \eqref{eqpsib1}, multiply by $\Lambda_x^s \psib^n/M$ and integrate in $x$ and $q$ to obtain
\begin{align*}
\frac12\ddt\normhbarl{\psib^n}s^2+\al_3\normhbardoth{\psib^n}s^2
&=-\iint_{\R^2\times D}\Lambda_x^s( \ub^n \cdot \nabla_x \psib^n)\frac{\Lambda_x^s\psib^n}M
- \iint_{\R^2\times D}\Lambda_x^s\nabla_q\cdot \left(\sigma(\ub^n) q {\psib}^n \right)\frac{\Lambda_x^s\psib^n}M\\
&\hskip 5cm -a\iint_{\R^2\times D}\Lambda_x^s\nabla_q \cdot \left[ \sigma(\ub^n) q M \right]\frac{\Lambda_x^s\psib^n}M\\
&\equiv I_1+I_2+I_3.
\end{align*}

\medskip

We bound first $I_1$. Let $[\Lambda_x^s,\ub^n]$ be the standard commutator defined by $[\Lambda_x^s,\ub^n]f=\Lambda_x^s(\ub^nf)-\ub^n\Lambda_x^sf$. Using that $\ub^n$ is divergence free, we can write
\begin{multline*}
I_1= -\iint_{\R^2\times D}\ub^n \cdot \nabla_x \Lambda_x^s\psib^n\frac{\Lambda_x^s\psib^n}M -\iint_{\R^2\times D}\dive_x[\Lambda_x^s,  \ub^n]\psib^n\frac{\Lambda_x^s\psib^n}M
=-\iint_{\R^2\times D}\dive_x[\Lambda_x^s,  \ub^n]\psib^n\frac{\Lambda_x^s\psib^n}M\\
\leq \normlbarl{\Lambda_x^s\psib^n}\normlbarl{\dive_x[\Lambda_x^s,  \ub^n]\psib^n}
\leq \normlbarl{\Lambda_x^s\psib^n} \bigl\|\|[\Lambda_x^s,  \ub^n]\psib^n\|_{H^1(\R^2)}\bigr\|_{L^2_M}\\
\leq C \normlbarl{\Lambda_x^s\psib^n} \bigl\|\|\ub^n\|_{H^{s+1}(\R^2)}\|\psib^n\|_{H^s(\R^2)}\bigr\|_{L^2_M}
=C\|\ub^n\|_{H^{s+1}(\R^2)}\normhbarl {\psib^n}s^2\\
\leq C\|\A^{\frac{1+s}2}u^n\|_{L^2(\Omega)}\normhbarl{\psib^n}s^2
\end{multline*}
where we used the embedding $H^s\subset L^\infty$, the classical commutator estimates, see \cite[Section 3.6]{Tay91}, and relation \eqref{boundubar}. 

\medskip

Next, we write
\begin{multline*}
I_2=\iint_{\R^2\times D}\Lambda_x^s\left(\sigma(\ub^n) q {\psib}^n \right)\cdot\nabla_q\Bigl(\frac{\Lambda_x^s\psib^n}M\Bigr)  
\leq C\normhbardoth{\psib^n}s \bigl\|\|\sigma(\ub^n) q {\psib}^n\|_{H^s(\R^2)}\bigr\|_{L^2_M}\\
\leq C\normhbardoth{\psib^n}s  \bigl\|\|\sigma(\ub^n)\|_{H^s(\R^2)} \|{\psib}^n\|_{H^s(\R^2)}\bigr\|_{L^2_M}
\leq C\normhbardoth{\psib^n}s\normhbarl{\psib^n}s\|\A^{\frac{1+s}2}u^n\|_{L^2(\Omega)}.
\end{multline*}

\medskip

We now make an integration by parts in $I_3$ and bound as follows
\begin{multline*}
I_3=a\iint_{\R^2\times D}\Lambda_x^s\left[ \sigma(\ub^n) q M \right]\cdot\nabla_q\bigl(\frac{\Lambda_x^s\psib^n}M\bigr)
\leq a\|  \Lambda_x^s\sigma(\ub^n)\|_{L^2(\R^2)}\|q\sqrt M\|_{L^2(D)}\normhbardoth{\psib^n}s\\
\leq \frac C{R^2}\|\A^{\frac{1+s}2}u^n\|_{L^2(\Omega)}\normhbardoth{\psib^n}s
\end{multline*}
where we used \eqref{valoarea} and \eqref{boundel}.

From relation \eqref{intpsibn} we have that $\int_D\Lambda_x^s\psib^n\dq\equiv0$. Relation \eqref{poincare} together withe the above estimates imply \eqref{psihsgen}. The corotational case \eqref{psihscor} also follows since in this case $I_3=0$.
\end{proof}

\paragraph{\large $H^s$ uniform bounds in the general case.}  We consider here the general case $\sigma(u)=\nabla u$. Let us first state the following remark.
\begin{remark}\label{grad2}
One can easily check that, given four strictly positive constants $A_1,A_2,A_3,A_4$ we have the following property: there exists some $\omega>0$ such that
\begin{equation*}
  A_1X^2+\omega A_2Y^2\geq A_3XY+\omega A_4XY\qquad \forall\ X,Y
\end{equation*}
if and only if $A_1A_2\geq A_3A_4$. Moreover, if the later is true then one can choose $\omega=\frac{2A_1A_2-A_3A_4}{A_4^2}$ which is of the same order as $A_1A_2/A_4^2$.
\end{remark}

We impose now that the condition above holds true with constants 
\begin{equation*}
A_1=\frac{\al_1}2, \quad A_2=\al_3, \quad A_3=C_1\frac{\al_2}{\delta^2}, \quad A_4=\frac{C_2}{R^2},  
\end{equation*}
that is we impose that
\begin{equation}\label{condcoeff}
\al_1\al_3\delta^2R^2\geq 2C_1C_2\al_2.  
\end{equation}
Let $\omega$ be as in the previous remark, of the same order as $\al_1\al_3R^4$. Assume moreover that
\begin{equation}
  \label{hypdatan}
 \| \A^{\frac s2}u^n_0 \|^2_{L^2(\Omega)}+\omega \normhbarl{\psib^n_0}s^2<\min\bigl(\frac{\al_1^2}{16C_1^2},\frac{\al_1\al_3}{4C_2^2}\bigr).  
\end{equation}
We multiply \eqref{psihsgen} by $\omega$ and add the result to \eqref{uhs}. After using Remark \ref{grad2} and recalling that $\psib^n\bigl|_{\Omega\times D}=\psi^n$ we obtain that
\begin{multline*}
  \ddt (\| \A^{\frac s2}u^n  \|^2_{L^2(\Omega)}+\omega \normhbarl{\psib^n}s^2)
+  \frac{\al_1}2 \| \A^{\frac{1+s}2}u^n \|^2_{L^2(\Omega)} +\omega\al_3\normhbardoth{\psib^n}s^2\\
\leq C_1\| \A^{\frac s2}u^n\|_{L^2(\Omega)}\| \A^{\frac {s+1}2}u^n\|_{L^2(\Omega)}^2  +\omega C_2 \|\A^{\frac{1+s}2}u^n\|_{L^2(\Omega)}  \normhbarl{\psib^n}s\normhbardoth{\psib^n}s.
\end{multline*}

Let $T_0$ be the first time such that
\begin{equation}\label{deftz}
 \| \A^{\frac s2}u^n (T_0) \|^2_{L^2(\Omega)}+\omega \normhbarl{\psib^n(T_0)}s^2=\min\bigl(\frac{\al_1^2}{16C_1^2},\frac{\al_1\al_3}{4C_2^2}\bigr).
\end{equation}
Then, for $t\in[0,T_0]$, we have that
\begin{equation*}
 \| \A^{\frac s2}u^n (t) \|\leq  \frac{\al_1}{4C_1}
\end{equation*}
so that
\begin{equation*}
C_1\| \A^{\frac s2}u^n(t)\|_{L^2(\Omega)}\| \A^{\frac {s+1}2}u^n(t)\|_{L^2(\Omega)}^2  \leq   \frac{\al_1}4\| \A^{\frac{1+s}2}u^n(t) \|^2_{L^2(\Omega)}.
\end{equation*}
We also have that
\begin{equation*}
 \normhbarl{\psib^n(t)}s\leq \frac1{2C_2}\sqrt{\frac{\al_1\al_3}{\omega}} 
\end{equation*}
so
\begin{align*}
\omega C_2 \|\A^{\frac{1+s}2}u^n\|_{L^2(\Omega)}  \normhbarl{\psib^n}s\normhbardoth{\psib^n}s
&\leq \frac{\sqrt{\al_1\al_3\omega}}{2}  \|\A^{\frac{1+s}2}u^n\|_{L^2(\Omega)} \normhbardoth{\psib^n}s\\
&\leq \frac{\al_1}8 \| \A^{\frac{1+s}2}u^n \|^2_{L^2(\Omega)} +\frac{\omega\al_3}2\normhbardoth{\psib^n}s^2
\end{align*}

We deduce from the above relations that,  for $t\in[0,T_0]$,
\begin{equation}\label{si1}
  \ddt (\| \A^{\frac s2}u^n  \|^2_{L^2(\Omega)}+\omega \normhbarl{\psib^n}s^2)+\frac{\al_1}8 \| \A^{\frac{1+s}2}u^n \|^2_{L^2(\Omega)} +\frac{\omega\al_3}2\normhbardoth{\psib^n}s^2\leq0 
\end{equation}
which implies that
\begin{equation*}
 \| \A^{\frac s2}u^n(T_0)  \|^2_{L^2(\Omega)}+\omega \normhbarl{\psib^n(T_0)}s^2 \leq \| \A^{\frac s2}u^n_0 \|^2_{L^2(\Omega)}+\omega \normhbarl{\psib^n_0}s^2<\min\bigl(\frac{\al_1^2}{16C_1^2},\frac{\al_1\al_3}{4C_2^2}\bigr).
\end{equation*}
This contradicts \eqref{deftz}. Therefore the time $T_0$ cannot exist, so
\begin{equation*}
  \| \A^{\frac s2}u^n(t)  \|^2_{L^2(\Omega)}+\omega \normhbarl{\psib^n(t)}s^2 <\min\bigl(\frac{\al_1^2}{16C_1^2},\frac{\al_1\al_3}{4C_2^2}\bigr)\qquad\forall t\geq0
\end{equation*}
and relation \eqref{si1} must hold true for all $t\geq0$.

We state the result proved in this paragraph in the following proposition.
\begin{proposition}\label{propgen}
Suppose that $\sigma(u)=\nabla u$ and that the material coefficients verify relation \eqref{condcoeff}.  Moreover assume that
\begin{equation}
  \label{condata}
 \| u_0 \|^2_{H^s(\Omega)}+C^2_0\omega \normhl{\psi_0}s^2<\min\bigl(\frac{\al_1^2}{16C_1^2},\frac{\al_1\al_3}{4C_2^2}\bigr).  
\end{equation}
Then the sequence $u^n$ is uniformly bounded in the space $L^\infty(\R_+;H^s(\Omega))\cap L^2(\R_+;H^{s+1}(\Omega))$ and the sequence $\psi^n$    is uniformly bounded in $L^\infty(\R_+;\hl s )\cap L^2(\R_+;\hdoth s)$.
\end{proposition}
\begin{proof}
It suffices to show that \eqref{condata} implies \eqref{hypdatan} for $n$ sufficiently large. This follows at once from \eqref{cond1}, \eqref{cond2} and \eqref{cond3}.  
\end{proof}

\paragraph{\large $H^s$ uniform bounds in the corotational case.}  We consider now the corotational case $\sigma(u)=\nabla u-(\nabla u)^t$. Denoting
\begin{equation*}
f_1= \| \A^{\frac s2}u^n\|^2_{L^2(\Omega)},\quad
f_2= \| \A^{\frac {s+1}2}u^n\|^2_{L^2(\Omega)},\quad
g_1= \normhbarl{\psib^n}s^2\quad\text{and}\quad
g_2= \normhbardoth{\psib^n}s^2
\end{equation*}
and recalling that $\psib^n\bigl|_{\Omega\times D}=\psi^n$, we have from relations \eqref{uhs} and \eqref{psihscor} that
\begin{gather*}
f_1'+\al_1 f_2\leq C_1hf_1+\frac{C_1\al_2}{\delta^2}\sqrt{f_2g_2}\\
\intertext{and that}
g_1'+2\al_3 g_2\leq C_2\sqrt{g_1f_2g_2}.
\end{gather*}

Using the following two bounds
\begin{equation*}
  C_2\sqrt{g_1f_2g_2}\leq \al_3 g_2+\frac{C_2^2}{4\al_3}g_1f_2
\quad\text{ and }\quad
  \frac{C_1\al_2}{\delta^2}\sqrt{f_2g_2}\leq \frac{\al_1}2 f_2+ \frac{C^2_1\al^2_2}{2\al_1\delta^4}g_2
\end{equation*}
we infer that
\begin{gather}
f_1'+\frac{\al_1}2 f_2\leq C_1hf_1+\frac{C_1^2\al_2^2}{2\al_1\delta^4}g_2\label{eqf1}\\
\intertext{and that}
g_1'+\al_3 g_2\leq \frac{C^2_2}{4\al_3}g_1f_2.\label{eqg1}
\end{gather}

Let $\ep$ be a small enough constant to be chosen later but such that
\begin{equation}\label{ep1}
g_1(0)<\ep.  
\end{equation}
Let $T_0$ be the first time such that
\begin{equation}\label{g1t}
g_1(T_0)=\ep.  
\end{equation}
We have that $g_1(t)<\ep$ for all $t\in[0,T_0)$. Using this in \eqref{eqg1} and integrating in time implies that for all $t\in[0,T_0]$
\begin{equation}\label{boundg1}
g_1(t)+\al_3\int_0^tg_2\leq  g_1(0)+\frac{C_2^2\ep}{4\al_3}\int_0^tf_2. 
\end{equation}

Multiplying \eqref{eqf1} by $e^{-C_1\int_0^th}$ and integrating in time results in
\begin{equation*}
f_1(t)+\frac{\al_1}2\int_0^t f_2(s)e^{C_1\int_s^th}ds\leq f_1(0)  e^{C_1\int_0^th}+\frac{C^2_1\al^2_2}{2\al_1\delta^4}\int_0^t g_2(s)e^{C_1\int_s^th}ds
\end{equation*}
We use now the estimate \eqref{boundg1} above. We infer
\begin{equation*}
 f_1(t)+\frac{\al_1}2\int_0^t f_2 \leq \bigl[f_1(0)+\frac{C^2_1\al_2^2}{2\al_1\al_3\delta^4}g_1(0)\bigr]  e^{C_1\int_0^\infty h}
+\frac{C^2_1C^2_2\al_2^2\ep}{8\al_1\al_3^2\delta^4}e^{C_1\int_0^\infty h}\int_0^t f_2.
\end{equation*}
We now add the following assumption on $\ep$:
\begin{equation}\label{ep2}
\frac{C^2_1C^2_2\al_2^2\ep}{8\al_1\al_3^2\delta^4}e^{C_1\int_0^\infty h}\leq\frac{\al_1}4
\end{equation}
Assuming that this is true, we further obtain that for all $t\in[0,T_0]$
\begin{equation*}
 f_1(t)+\frac{\al_1}4\int_0^t f_2 \leq \bigl[f_1(0)+\frac{C^2_1\al_2^2}{2\al_1\al_3\delta^4}g_1(0)\bigr]  e^{C_1\int_0^\infty h}.
\end{equation*}
Going back to \eqref{eqg1}, ignoring the second term on the left-hand side and using the Gronwall lemma implies now that 
\begin{equation}
  \label{boundg1final}
 g_1(t)\leq  g_1(0)\exp\Bigl\{\frac{C_2^2}{\al_1\al_3}\bigl[f_1(0)+\frac{C^2_1\al_2^2}{2\al_1\al_3\delta^4}g_1(0)\bigr]  e^{C_1\int_0^\infty h}\Bigr\}
\end{equation}
for all $t\in[0,T_0]$. If we further assume that
\begin{equation}\label{ep3}
 g_1(0)\exp\Bigl\{\frac{C_2^2}{\al_1\al_3}\bigl[f_1(0)+\frac{C^2_1\al_2^2}{2\al_1\al_3\delta^4}g_1(0)\bigr]  e^{C_1\int_0^\infty h}\Bigr\}<\ep 
\end{equation}
then we observe that setting $t=T_0$ in \eqref{boundg1final} contradicts \eqref{g1t}. We conclude that under the hypothesis \eqref{ep1}, \eqref{ep2} and \eqref{ep3} the time $T_0$ cannot exist, so all the previous relations hold true for all times $t\geq0$. Clearly \eqref{ep1} is implied by \eqref{ep3}. Recalling \eqref{inegh} we therefore observe that there exists some $\ep$ verifying \eqref{ep1}, \eqref{ep2} and \eqref{ep3} if we have that
\begin{multline}\label{hypdatan1}
g_1(0)\exp\Bigl\{\frac{C_2^2}{\al_1\al_3}\bigl[f_1(0)+\frac{C^2_1\al_2^2}{2\al_1\al_3\delta^4}g_1(0)\bigr]  
e^{C_1\al_1^{-\frac4s}\bigl( \|u_0 \|_{L^2(\Omega)}^2 + \al_4 \normll{\psi_0}^2\bigr)^{\frac2s}}\Bigl\} \\
<\frac{2\al_1^2\al_3^2\delta^4}{C_1^2C_2^2\al_2^2} \exp\Bigl[-C_1\al_1^{-\frac4s}\bigl( \|u_0 \|_{L^2(\Omega)}^2 + \al_4 \normll{\psi_0}^2\bigr)^{\frac2s}\Bigl]. 
\end{multline}

We state the result proved in this paragraph in the following proposition.
\begin{proposition}\label{propcor}
Suppose that $\sigma(u)=(\nabla u-\nabla u)^t$. There exists a constant $C=C(\Omega,s)$ such that if 
\begin{multline}\label{condata1}
\normhl{\psi_0}{s}\exp\Bigl\{\frac{C}{\al_1\al_3}\bigl[\|u_0\|^2_{H^s(\Omega)}+\frac{C\al_2^2}{\al_1\al_3\delta^4}\normhl{\psi_0}{s}^2\bigr]  
e^{C\al_1^{-\frac4s}\bigl( \|u_0 \|_{L^2(\Omega)}^2 + \al_4 \normll{\psi_0}^2\bigr)^{\frac2s}}\Bigl\} \\
<\frac{\al_1^2\al_3^2\delta^4}{C\al_2^2} \exp\Bigl[-C\al_1^{-\frac4s}\bigl( \|u_0 \|_{L^2(\Omega)}^2 + \al_4 \normll{\psi_0}^2\bigr)^{\frac2s}\Bigl]
\end{multline}
then the sequence $u^n$ is uniformly bounded in the space $L^\infty(\R_+;H^s(\Omega))\cap L^2(\R_+;H^{s+1}(\Omega))$ and the sequence $\psi^n$    is uniformly bounded in $L^\infty(\R_+;\hl s )\cap L^2(\R_+;\hdoth s)$.
\end{proposition}
\begin{proof}
It suffices to show that \eqref{condata1} implies \eqref{hypdatan1} for $n$ sufficiently large. This  follows at once from \eqref{cond1}, \eqref{cond2} and \eqref{cond3}.  
\end{proof}
\begin{remark}\label{remcor}
It is not difficult to see that there exists some constant $K=K(\Omega,s,\al_1,\al_2,\al_3,\al_4,\delta)$ such that condition  \eqref{condata1} is implied by the following condition:
\begin{equation*}
\normhl{\psi_0}{s}\leq \exp\Bigl[-K(1+\|u_0\|_{H^s(\Omega)}) e^{K\|u_0 \|_{L^2(\Omega)}^{\frac4s}}\Bigr]. 
\end{equation*}
\end{remark}

\paragraph{\large End of the proof.}
Clearly the hypothesis of Proposition \ref{propgen} is implied by that of Theorem \ref{thgen}, and the hypothesis of Proposition \ref{propcor} is implied by that of Theorem \ref{thcor} (see also Remark \ref{remcor}). Therefore, under the hypothesis of Theorem \ref{thgen} in the general case and under the hypothesis of Theorem \ref{thcor} in the corotational case, we have that the sequence $u^n$ is uniformly bounded in the space $L^\infty(\R_+;H^s(\Omega))\cap L^2(\R_+;H^{s+1}(\Omega))$ and the sequence $\f^n$    is uniformly bounded in $L^\infty(\R_+;\hl s )\cap L^2(\R_+;\hdoth s)$. Using the equations of $u^n$ and $\f^n$ \eqref{fas5} and \eqref{fas6} this immediately implies some time-derivative estimates for   $u^n$ and $\f^n$. A standard compactness argument allows to pass to the limit and find a solution $(u,\f)$ of \eqref{fequ}--\eqref{feqpsi} such that $u\in L^\infty(\R_+;H^s(\Omega))\cap L^2(\R_+;H^{s+1}(\Omega))$ and $\f\in L^\infty(\R_+;\hl s )\cap L^2(\R_+;\hdoth s)$. We skip the details since this is very classical and straightforward. Next, we clearly have that $u$ is divergence free and tangent to the boundary. Passing to the limit in \eqref{intfn} shows that that relation holds true with $f^n$ replaced by $f$. Finally, the equation \eqref{fas6} preserves the sign if the initial data is single-signed so $f^n\geq0$ which implies in turn that $f\geq0$. This completes the proof of the existence of the solution.

The uniqueness of solutions is obvious and follows by making energy estimates on the difference between two solutions. If $(u_1,\f_1)$ and $(u_2,\f_2)$ are two solutions with the same data, then  we multiply the difference of the equations of $u_1$ and $u_2$ by
$u_1-u_2$ and the difference of the equations of $\f_1$ and $\f_2$ by $\frac{\f_1-\f_2}M$ and add the two resulting relations. Uniqueness follows easily from the Gronwall inequality using that $u\in L^2(\R_+;Lip)$ and the linearity in the $q$ variable of the equation of $\f$. This is very standard so we skip the details. Remark that even though we show that \eqref{feqpsi} holds true in the sense of distributions, \textit{i.e.} \eqref{feqpsi} can be multiplied by test functions which are compactly supported in the $q$ variable, it can in fact be multiplied by functions which are $H^1_M$ in the $q$ variable. This follows from a density argument using that $C^\infty_0(D)$ is dense in $H^1_M(D)$ as was proved in \cite{Mas08a}. This completes the proofs of Theorems \ref{thgen} and \ref{thcor}.

\section{Final remarks}\label{finalremarks}

First, we would like to explain here why the condition on the coefficients \eqref{condcoefftheo} is necessary in the general case on bounded domains. We will observe that a certain cancellation that occurs in the case without boundary does not work anymore in the presence of boundaries. When making $H^m$ estimates on $u$ and $\hl m$ estimates on $\psi=f-aM$ we apply $\partial^\al$ to the equation of $u$ given in \eqref{fequ} and multiply by $\partial^\al u$, we apply $\partial^\al$ to the equation of $\psi$ given in \eqref{as6} (where we dropped the superscript $n$)  and multiply by $\frac{\partial^\al \psi}{2\delta a M}$ and we add the two resulting relations. We get the following right-hand side: 
\begin{equation*}
\int_\Omega \Bigl[\nabla_x \cdot \int_D \dfrac{q\otimes q}{1-|q|^2}\partial_x^\al \f\dq \Bigr]\cdot\partial_x^\al u- \frac1{2\delta}\iint_{\Omega\times D}\nabla_q\cdot\left[\nabla \partial_x^\al u \, q M \right]\frac{\partial_x^\al\f}M.
\end{equation*}
In the case of a domain without boundary, making an integration by parts implies, after some calculations, that the term above vanishes. But in the case of a domain with boundary, the boundary terms do not vanish. Moreover, due to the presence of the pressure in the equation of $u$, the first term above should have the Leray projector $\PP$ in front of $\nabla_x\cdot$ making the validity of this identity even more unlikely in presence of boundaries. Since the term above does not vanish anymore, we need to be able to say that it is small (negligible compared to others) and this in turn requires the smallness condition \eqref{condcoefftheo}.

\bigskip

Second, we would like to explain why the restriction $1<s<\frac32$ is necessary. In order to be able to control the equation on $\f$ we basically need Lipschitz regularity for $u$. If we assume that the initial velocity belongs to $H^s(\Omega)$, then the standard regularity for $u$ obtained through energy estimates is $L^2_tH^{s+1}_x$. To get Lipschitz regularity in $x$ we therefore need to assume that $s>1$. On the other hand, when making the same $H^s$ estimates on $u^n$ we are led to applying the  projector $\PP_n$ to the equation of $u^n$ \eqref{equn} and to estimate the right-hand side in $H^{s-1}$. This requires the  projection $\PP_n$ to be bounded in $H^{s-1}$ which implies $s-1<\frac12$ so $s<\frac32$. This explains why the condition $1<s<\frac32$ is required. We would also like to point out that in dimension three, the first requirement that $u$ to be Lipschitz implies $s>\frac32$ while the second requirement does not change leading to contradictory assumptions. This means that our approach does not work in dimension three.

\paragraph{\large Acknowledgements.}

One of us, L.I.P., is grateful to  Professor Robert Byron
Bird, University of Wisconsin, Madison, for kind support in the past.


\def\cprime{$'$} \def\cprime{$'$} \def\cydot{{\l}eavevmode\raise.4ex\hbox{.}}
  \def\cprime{$'$} \def\cprime{$'$}
  \def\polhk\#1{\setbox0=\hbox{\#1}{{\o}oalign{\hidewidth
  {\l}ower1.5ex\hbox{`}\hidewidth\crcr\unhbox0}}}

\adrese

\end{document}